\documentclass[12pt]{amsart}
\usepackage{amssymb}
\usepackage{amsfonts}
\usepackage{color}

\theoremstyle{plain}

\newtheorem{corollary}{Corollary}

\newtheorem{definition}{Definition}

\newtheorem{proposition}{Proposition}
\newtheorem{remark}{Remark}

\newtheorem{theorem}{Theorem}
\numberwithin{equation}{section}

\begin{document}
\title[The affine ensemble]{The affine ensemble: determinantal point
processes associated with the $ax+b$ group}
\author{Lu\'{\i}s Daniel Abreu}
\email{abreuluisdaniel@gmail.com}
\address{NuHAG, Faculty of Mathematics, University of Vienna,
Oskar-Morgenstern-Platz 1, A-1090, Vienna, Austria}
\author{Peter Balazs}
\email{peter.balazs@oeaw.ac.at}
\address{Acoustics Research Institute, Vienna 1040, Austria}
\author{Smiljana Jak\v si\'c}
\email{smiljana.jaksic@sfb.bg.ac.rs}
\address{Faculty of Forestry, University of Belgrade, Kneza Vi seslava 1,
11000, Belgrade, Serbia}
\thanks{This work was supported by the Austrian ministry BMBWF through the
WTZ/OeAD-projects SRB 01/2018 "ANACRES - Analysis and Acoustics Research"
and MULT 10/2020 "Time-Frequency representations for function spaces -
Tireftus and FWF project `Operators and Time-Frequency Analysis' P 31225-N32.%
}

\begin{abstract}
We introduce the affine ensemble, a class of determinantal point processes
(DPP) in the half-plane $\mathbb{C}^{+}$ associated with the $ax+b$ (affine)
group, depending on an admissible Hardy function $\psi $. We obtain the
asymptotic behavior of the variance, the exact value of the asymptotic
constant, and non-asymptotic upper and lower bounds for the variance on a
compact set $\Omega \subset $ $\mathbb{C}^{+}$. As a special case one
recovers the DPP related to the weighted Bergman kernel. When $\psi $ is
chosen within a finite family whose Fourier transform are Laguerre
functions, we obtain the DPP associated to hyperbolic Landau levels, the
eigenspaces of the finite spectrum of the Maass Laplacian with a magnetic
field.
\end{abstract}

\maketitle

\section{Introduction}

Determinantal point processes (DPPs) are random point distributions with
negative correlations between points determined by the reproducing kernel of
some Hilbert space, usually called \emph{the correlation kernel}. Because of
the repulsion inherent of the model, DPPs are convenient to describe
physical systems with charged-liked particles, to distribute random
sequences of points in selected regions while avoiding clustering, to
promote diversity in selection algorithms for machine learning \cite{ML}, or
to improve the rate of convergence in Monte Carlo methods \cite{BH}. DPPs
have been introduced by Odile Macchi to model fermion distributions \cite%
{Macchi}.

In this paper we introduce and study some aspects of \emph{the affine
ensemble,} a family of determinantal point processes on the complex upper
half-plane $\mathbb{C}^{+}$, defined in terms of a representation of the $%
ax+b$ group acting on a vector $\psi \in H^{2}(\mathbb{C}^{+})$, the Hardy
space in the upper half plane. This can be seen as a geometric hyperbolic or
algebraic non-unimodular analogue of the Weyl-Heisenberg ensemble \cite%
{abgrro17,APRT}, a family of planar euclidean determinantal point processes
associated with the (unimodular) Weyl-Heisenberg group. In terms of the
representation $\pi (z)\psi (t):=s^{-\frac{1}{2}}\psi (s^{-1}(t-x))$, the
kernel of the affine ensemble is defined for $\psi $ such that $\ \Vert \psi
\Vert _{2}=1$ and $\left\Vert \mathcal{F}\psi \right\Vert _{L^{2}(%
\mathbb{R}
^{+},t^{-1}dt)}<\infty $, as a normalizing constant times 
\begin{equation}
k_{\psi }(z,w)=\langle \pi (w)\psi ,\pi (z)\psi \rangle _{H^{2}\left( 
\mathbb{C}^{+}\right) }\text{.}  \label{Kg1}
\end{equation}%
Using the Fourier transform isomorphism $\mathcal{F}:H^{2}(\mathbb{C}%
^{+})\rightarrow L^{2}(0,\infty )$, the kernel (\ref{Kg1}) can be written in
a more convenient form, for $z=x+is,w=x^{\prime }+is^{\prime }\in \mathbb{C}%
^{+}$, 
\begin{equation}
k_{\psi }(z,w)=(ss^{\prime })^{\frac{1}{2}}\int_{0}^{\infty }e^{-ix^{\prime
}\xi }(\mathcal{F}\psi )(s^{\prime }\xi )\overline{e^{-ix\xi }(\mathcal{F}%
\psi )(s\xi )}d\xi \text{.}  \label{Kg2}
\end{equation}

By selecting special functions $\psi $, a number of processes arise as
special cases, automatically inheriting all properties of the affine
ensemble. In this paper we will consider only examples with $PSL(2,\mathbb{R}%
)$ invariance, corresponding to invariance under the linear fractional
transformations of the half plane $\mathbb{C}^{+}$ (dilations, rotations and
translations). Consider the mother wavelets chosen from the family $\{\psi
_{n}^{\alpha }\}_{n\in \mathbb{N}_{0}}$, $\alpha >0$, 
\begin{equation}
(\mathcal{F}\psi _{n}^{\alpha })(\xi ):=\xi ^{\frac{\alpha }{2}}e^{-\xi
}L_{n}^{\alpha }(2\xi ),\quad \xi >0\text{,}  \label{mother}
\end{equation}%
where $L_{n}^{\alpha }$ denotes the generalized Laguerre polynomials 
\begin{equation*}
L_{n}^{\alpha }(t)=\frac{t^{-\alpha }e^{t}}{n!}\left( \frac{d}{dt}\right)
^{n}(e^{-t}t^{\alpha +n}),\quad t>0\text{.}
\end{equation*}%
and the following projective unitary group representation of $PSL(2,\mathbb{R%
})$ on $L^{2}(\mathbb{C}^{+},\mu )$:%
\begin{equation}
\widehat{\tau }_{n}^{\alpha }\left( \frac{az+b}{cz+d}\right) F(z):=\left( 
\frac{|cz+d|}{cz+d}\right) ^{2n+\alpha +1}F\left( \frac{az+b}{cz+d}\right) 
\text{,}  \label{eq:tau-n-widehat}
\end{equation}%
where \ $a,b,c,d$ are real numbers such that $\det \left[ 
\begin{array}{cc}
a & b \\ 
c & d%
\end{array}%
\right] \neq 0$.\ Then $\widehat{\tau }_{n}^{\alpha }$ leaves the Hilbert
space with reproducing kernel $k_{\psi _{n}^{\alpha }}(z,w)$ invariant. This
has been shown in \cite{Affine}$\ $and also that, essentially, the choice (%
\ref{mother}) is the only leading to spaces invariant under representations
of the form \eqref{eq:tau-n-widehat}, if we assume mild reasonable
restrictions on $\psi $ (see Theorem 3 in \cite{Affine}). Moreover, for $%
\psi $ within the family (\ref{mother}) and the parameter $\alpha =2(B-n)-1$
we obtain reproducing kernels associated with the eigenspaces of the pure
point spectrum of the Maass Laplacian with weight $B$ \cite{Maass,Comtet,AoP}%
: 
\begin{equation*}
H_{B}:=s^{2}\left( \frac{\partial ^{2}}{\partial x^{2}}+\frac{\partial ^{2}}{%
\partial s^{2}}\right) -2iBs\frac{\partial }{\partial x}
\end{equation*}

The last part of this paper will be devoted to these special cases. The pure
point spectrum eigenspaces of the Maass Laplacian $H_{B}$ have been used in 
\cite{Comtet,AoP} to model the formation of higher Landau levels in the
hyperbolic plane. A physical model, put forward by Alain Comtet in \cite%
{Comtet}, describes a situation where the number of levels is constrained to
be a finite number, depending on the strength of the magnetic field $B$,
which must exceed a lower bound for their existence (the magnetic field has
to be strong enough to capture the particle in a closed orbit). The
connection to analytic wavelets was suggested by the characterization of
hyperbolic Landau coherent states \cite{Mouayn} and has been implicit used
in \cite{AoP} and more recently in \cite{Affine}.

It is reasonable to expect interesting examples arising from other special
choices, namely those leading to the polyanalytic structure discovered by
Vasilevski \cite{VasiBergman} (see \cite{SF,Hutnik,IEOT} for the special
choices leading to polyanalytic spaces) but we will not explore this
direction in the present paper. One of our motivations for this research was
the scarceness of examples on hyperbolic DPPs. Besides the celebrated case
studied by Peres and Vir\'{a}g \cite{PeresVirag}, we only found the higher
Landau levels DPP on the disc studied recently by Demni and Lazag in \cite%
{HyperbolicDPP}, which strongly influenced the current paper. The affine
ensemble contains an uncountable number of examples as special cases, of
which we only explore a very few.

The paper is organized as follows. The next section contains the required
background on analytic wavelets and hyperbolic geometry. The third section
is the core of the paper, where the affine ensemble is defined and the main
results are proved. Section 4 specializes the results to the class of
Maass-Landau ensembles and the calculations are detailed in the last section
of the paper, as an appendix.

\section{Background}

\subsection{The continuous analytic wavelet transform}

We will use the basic notation for $H^{2}(%
\mathbb{C}
^{+})$, the Hardy space in the upper half plane, of analytic functions in $%
\mathbb{C}
^{+}$ with the norm 
\begin{equation*}
\left\Vert f\right\Vert _{H^{2}(%
\mathbb{C}
^{+})}=\text{ }\sup_{0<s<\infty }\int_{-\infty }^{\infty }\left\vert
f(x+is)\right\vert ^{2}dx<\infty \text{.}
\end{equation*}%
To simplify the computations it is often convenient to use the equivalent
definition (since the Paley-Wiener theorem \cite{DGM} gives $\mathcal{F}%
(H^{2}(\mathbb{C}^{+}))=L^{2}(0,\infty )$) 
\begin{equation*}
H^{2}(\mathbb{C}^{+})=\left\{ f\in L^{2}(\mathbb{R}):(\mathcal{F}f)(\xi )=0%
\text{ for almost all }\xi <0\right\} \text{.}
\end{equation*}%
Consider the $ax+b$ group (see \cite[Chapter 10]{Charly} for the listed
properties) $G\sim \mathbb{R}\times \mathbb{R}^{+}\sim $ $\mathbb{C}^{+}$
with the multiplication%
\begin{equation*}
(x,s)\cdot (x^{\prime },s^{\prime })=(x+sx^{\prime },ss^{\prime })\text{.}
\end{equation*}%
The identification $G\sim $ $\mathbb{C}^{+}$\ is done by setting $(x,s)\sim
x+is$. The neutral element of the group is $(0,1)\sim i$ and the inverse
element is given by $(x,s)^{-1}=(-\frac{x}{s},\frac{1}{s})\equiv -\frac{x}{s}%
+\frac{i}{s}$. The $ax+b$ group is not\ unimodular, since the left Haar
measure on $G$ is $\frac{dxds}{s^{2}}\,\ $and the right Haar measure $G$ is $%
\frac{dxds}{s}$. The left Haar measure of a set $\Omega \subseteq G$, 
\begin{equation*}
|\Omega |=\int_{\Omega }\frac{dxds}{s^{2}}\text{, }
\end{equation*}%
coincides, under the identification of the $ax+b$ group with $\mathbb{C}^{+}$%
, with the hyperbolic measure%
\begin{equation*}
|\Omega |=\left\vert \Omega \right\vert _{h}:=\int_{\Omega }s^{-2}\,d\mu _{%
\mathbb{C}^{+}}(z)\text{,}
\end{equation*}%
where $d\mu _{\mathbb{C}^{+}}(z)\ $is the Lesbegue measure in $\mathbb{C}^{+}
$. We will write 
\begin{equation}
d\mu ^{+}(z)=(\mathrm{Im}\,z)^{-2}d\mu _{\mathbb{C}^{+}}(z)\text{.}
\label{measure}
\end{equation}

For every $x\in \mathbb{R}$ and $s\in \mathbb{R}^{+}$, define the
translation $T_{x}$ by $T_{x}f(t)=f(t-x)$ and the dilation $D_{s}f(t)=\frac{1%
}{\sqrt{s}}f(t/s)$. Let $z=x+is\in 
\mathbb{C}
^{+}$ and define the representation, for $\psi \in H^{2}(\mathbb{C}^{+})$, 
\begin{equation}
\pi (z)\psi (t):=T_{x}D_{s}\psi (t)=s^{-\frac{1}{2}}\psi (s^{-1}(t-x))\text{.%
}  \label{representation}
\end{equation}%
{The theory of general wavelet transforms using group representations
requires the admissibility condition \cite{aliant1},%
\begin{equation*}
\int\limits_{G}\left\vert \left\langle \psi ,\pi (z)\psi \right\rangle
\right\vert ^{2}d\mu (z)<\infty \text{,}
\end{equation*}%
to construct square-integrable representations for general groups $G$ with
left Haar measures $\mu $. This will lead to an isometric transform thanks
to the orthogonality relations (\ref{ortogonalityrelations}) below. In} the
Weyl-Heisenberg representations used in \cite{abgrro17,APRT}, this follows
trivially from the square-integrability of $\psi ${\ only}. But here,{\ in
the affine case,} we need to take into account that the $ax+b$ group is not\
unimodular, and the different left and right Haar measures of the
representation require a further condition on the integrability of $\psi $,
which will be restricted to the class of functions such that 
\begin{equation}
\left\{ 
\begin{array}{c}
\psi \in H^{2}(\mathbb{C}^{+}) \\ 
0<2\pi \left\Vert \mathcal{F}\psi \right\Vert _{L^{2}(%
\mathbb{R}
^{+},t^{-1}dt)}^{2}=C_{\psi }<\infty 
\end{array}%
\right. \text{.}  \label{Adm_const}
\end{equation}%
Functions satisfying (\ref{Adm_const}) are called \emph{admissible }and the
constant $C_{\psi }$ is the \emph{admissibility constant}.{\ }Now, we have
an irreducible and unitary representation $\pi $ of the affine group on $%
H^{2}(\mathbb{C}^{+})$ \cite{aliant1}, defined in (\ref{representation}) for
an admissible $\psi $. {By this definition any admissible function is
automatically in the Hardy space and so the inner product considered for the
wavelet transform or the reproducing kernel is in $H^{2}(\mathbb{C}^{+})$. \ 
}\emph{The continuous analytic wavelet transform} of a function $f$ with
respect to a wavelet\ $\psi $ is defined, for every $z=x+is\in \mathbb{C}^{+}
$, as 
\begin{equation}
W_{\psi }f(z)=\left\langle f,\pi ({z})\psi \right\rangle _{H^{2}\left( 
\mathbb{C}^{+}\right) }\text{.}  \label{wavelet}
\end{equation}%
More explicitly, 
\begin{equation*}
W_{\psi }f(z)=\sup_{0<s<\infty }s^{-\frac{1}{2}}\int_{-\infty }^{\infty
}f(t)\psi (s^{-1}(t-x))dt\text{.}
\end{equation*}%
Using $\mathcal{F}(H^{2}(\mathbb{C}^{+}))=L^{2}(0,\infty )$, this can also
be written (and we will do it as a rule to simplify the calculations) as%
\begin{equation}
W_{\psi }f(z)=s^{\frac{1}{2}}\int_{0}^{\infty }f(\xi )e^{-ix\xi }(\mathcal{F}%
\psi )(s\xi )d\xi \text{.}  \label{anwav}
\end{equation}%
As proven recently in \cite{AnalyticWavelet}, $W_{\psi }f(z)$ only leads to
analytic (Bergman) phase spaces for a very special choice of $\psi $, but it
is common practice to call it in general continuous analytic wavelet
transform. The orthogonality relations 
\begin{equation}
\int_{\mathbb{C}^{+}}W_{\psi _{1}}f_{1}(x,s)\overline{W_{\psi _{2}}f_{2}(x,s)%
}d\mu ^{+}(z)=2\pi \left\langle \mathcal{F}\psi _{1},\mathcal{F}\psi
_{2}\right\rangle _{L^{2}(%
\mathbb{R}
^{+},t^{-1}dt)}\left\langle f_{1},f_{2}\right\rangle _{H^{2}\left( \mathbb{C}%
^{+}\right) }\text{,}  \label{ortogonalityrelations}
\end{equation}%
are valid for all $f_{1},f_{2}\in H^{2}\left( \mathbb{C}^{+}\right) $ and $%
\psi _{1},\psi _{2}\in H^{2}\left( \mathbb{C}^{+}\right) $ admissible. Then,
setting $\psi _{1}=\psi _{2}=\psi $\ and $f_{1}=f_{2}$\ in (\ref%
{ortogonalityrelations}), gives%
\begin{equation*}
\int_{\mathbb{C}^{+}}\left\vert W_{\psi }f(x,s)\right\vert ^{2}d\mu
^{+}(z)=C_{\psi }\left\Vert f\right\Vert _{H^{2}\left( \mathbb{C}^{+}\right)
}^{2}
\end{equation*}%
and the continuous wavelet transform provides an isometric inclusion $%
W_{\psi }:H^{2}\left( \mathbb{C}^{+}\right) \rightarrow L^{2}(\mathbb{C}^{+}%
\mathbf{,}d\mu ^{+})$. Setting $\psi _{1}=\psi _{2}=\psi $\ and $f_{2}=\pi
(z)\psi $ in (\ref{ortogonalityrelations}) then for every $f\in H^{2}\left( 
\mathbb{C}^{+}\right) $ one has 
\begin{equation}
W_{\psi }f(z)=\frac{1}{C_{\psi }}\int_{\mathbb{C}^{+}}W_{\psi }f(w)\langle
\pi (w)\psi ,\pi (z)\psi \rangle d\mu ^{+}(z),\quad z\in \mathbb{C}^{+}\text{%
.}  \label{eq:rep-eq}
\end{equation}%
Thus, the range of the wavelet transform 
\begin{equation*}
W_{\psi }\left( H^{2}\left( \mathbb{C}^{+}\right) \right) :=\{F\in L^{2}(%
\mathbb{C}^{+},\mu ^{+}):\ F=W_{\psi }f,\ f\in H^{2}\left( \mathbb{C}%
^{+}\right) \}
\end{equation*}%
is a reproducing kernel subspace of $L^{2}(\mathbb{C}^{+},\mu ^{+})$ with
kernel%
\begin{equation}
k_{\psi }(z,w)=\frac{1}{C_{\psi }}\langle \pi (w)\psi ,\pi (z)\psi \rangle
_{H^{2}\left( \mathbb{C}^{+}\right) }=\frac{1}{C_{\psi }}W_{\psi }\psi ({w}%
^{-1}.z),\quad \text{and}\quad k_{\psi }(z,z)=\frac{\Vert \psi \Vert _{2}^{2}%
}{C_{\psi }}\text{.}  \label{RepKern}
\end{equation}%
The Fourier transform $\mathcal{F}:H^{2}\left( \mathbb{C}^{+}\right)
\rightarrow L^{2}(0,\infty )$ can be used to simplify computations, since%
\begin{equation}
\left\langle \pi (w)\psi ,\pi (z)\psi \right\rangle _{H^{2}\left( \mathbb{C}%
^{+}\right) }=\left\langle \widehat{\pi (w)\psi },\widehat{\pi (z)\psi }%
\right\rangle _{L^{2}(\mathbb{R}^{+},dt)}=\left( ss^{\prime }\right) ^{\frac{%
1}{2}}\int_{0}^{\infty }\widehat{\psi }(s^{\prime }\xi )\overline{\widehat{%
\psi }\left( s\xi \right) }e^{i(x-x^{\prime })\xi }d\xi \text{.}
\label{kernelF}
\end{equation}

\subsection{Hyperbolic geometry}

We will need some elementary facts of hyperbolic geometry. The hyperbolic
metric in $\mathbb{C}^{+}$ is defined as \cite{Hmetric}%
\begin{equation}
d(z_{1},z_{2})=\log \frac{1+\varrho (z_{1},z_{2})}{1-\varrho (z_{1},z_{2})}%
=2\tanh ^{-1}\left( \varrho (z_{1},z_{2})\right) \text{,}  \label{HypMetr}
\end{equation}%
where $\varrho $ is the pseudohyperbolic metric in $\mathbb{C}^{+}$, 
\begin{equation*}
\varrho (z_{1},z_{2})=\left\vert \frac{z_{1}-z_{2}}{z_{1}-\overline{z_{2}}}%
\right\vert \text{.}
\end{equation*}%
The hyperbolic ball of center $z\in \mathbb{C}^{+}$ and radius $R<1$ is 
\begin{equation*}
D(z,R)=\left\{ w\in \mathbb{C}^{+}:d(w,z)<R\right\} \text{.}
\end{equation*}%
By direct computation it can be checked that $\varrho (z^{-1},i)=\varrho
(i,z)$\ and that $D(z^{-1},R)=D(z,R)$.

\section{The affine ensemble}

\subsection{The affine ensemble}

\emph{Determinantal Point Processes} (we simply list the concepts we are
using; for a complete definition see \cite[Chapter 4]{DetPointRand}) are
defined using an ambient space $\Lambda $, a Radon measure $\mu $ defined on 
$\Lambda $, and\ a reproducing kernel Hilbert space $\mathcal{H}$ contained
in $L^{2}(\mathbb{C}^{+}\mathbf{,}d\mu ^{+})$. The reproducing kernel of $%
\mathcal{H}$, $K\left( z,w\right) $, is the correlation kernel of the point
process $\mathcal{X}$. The $k$-point intensities are given by $\rho
_{k}(x_{1},...,x_{k})=\det \left( K(x_{i},x_{j})\right) _{1\leq i,j\leq k}$.
Given a set $\Omega \subset \mathbb{C}^{+}$, the $1$-point intensity gives
the expected number of points to be found in $\Omega $: 
\begin{equation*}
\mathbb{E}\left( \mathcal{X}(\Omega )\right) =\int_{\Omega }\rho _{1}\left(
z\right) d\mu ^{+}(z)=\int_{\Omega }K\left( z,z\right) d\mu ^{+}(z)\text{,}
\end{equation*}

The normalization of the kernel $K_{\psi }(z,w)$ in the following definition
is different from the one in (\ref{RepKern}) and is chosen so that $K_{\psi
}(z,z)=1$. Recall that $k_{\psi }(z,z)=\frac{\Vert \psi \Vert _{2}^{2}}{%
C_{\psi }}$.

\begin{definition}
The \textbf{affine} \textbf{ensemble }$\mathcal{X}_{\psi }$ associated with
an admissible function $\psi $ is the Determinantal Point Process with the
normalized correlation kernel 
\begin{equation}
K_{\psi }(z,w)=\frac{k_{\psi }(z,w)}{k_{\psi }(z,z)}=\frac{W_{\psi }\psi ({w}%
^{-1}\cdot z)}{\Vert \psi \Vert _{2}^{2}}=W_{\widetilde{\psi }}\widetilde{%
\psi }({w}^{-1}\cdot z)\text{,}  \label{affineKernel}
\end{equation}%
where $\widetilde{\psi }=\psi /\Vert \psi \Vert _{2}$.
\end{definition}

We will assume from now on $\Vert \psi \Vert _{2}=1$ (if this is not the
case, we will use the notation $\widetilde{\psi }=\psi /\Vert \psi \Vert
_{2} $). Then 
\begin{equation*}
K_{\psi }(z,w)=k_{\psi }(z,w)=W_{\psi }\psi ({w}^{-1}\cdot z)\text{, \ \ \ \
\ \ \ \ \ }K_{\psi }(z,z)=W_{\psi }\psi ({i})
\end{equation*}%
and, from (\ref{ortogonalityrelations}),%
\begin{equation}
\int_{\mathbb{C}^{+}}\left\vert W_{\psi }\psi ({w})\right\vert ^{2}d\mu
^{+}(w)=C_{\psi }\text{.}  \label{Cg}
\end{equation}

\subsection{Variance estimates}

We will consider an operator $T_{\Omega }$ acting on a function $f$ on the
range of $W_{\psi }$, which smooths out the energy of $f$ outside $\Omega $,
by first multiplication by $\mathbf{1}_{\Omega }$ and then projecting on the
range of $W_{\psi }$. Using the reproducing kernel property, $T_{\Omega }$
can be written as 
\begin{eqnarray}
(T_{\Omega }f)(z) &=&\int_{\Omega }f(w)K_{\psi }(z,w)d\mu ^{+}(w)
\label{Toeplitz} \\
&=&\int_{\Omega }f(z^{\prime })\int_{\Omega }K_{\psi }(z,z^{\prime })K_{\psi
}(z^{\prime },w)d\mu ^{+}(z^{\prime })d\mu ^{+}(w)\text{,}  \notag
\end{eqnarray}%
providing an extension of $T_{\Omega }$ to the whole $L^{2}(\mathbb{C}^{+}%
\mathbf{,}s^{-2}dxds)$ vanishing in the complement of the range of $W_{\psi
} $. By definition of $1$-point intensity, 
\begin{equation*}
\mathbb{E}\left( \mathcal{X}_{\psi }(\Omega )\right) =\int_{\Omega }K_{\psi
}\left( z,z\right) d\mu ^{+}(z)=trace\left( T_{\Omega }\right) =\left\vert
\Omega \right\vert _{h}\text{.}
\end{equation*}%
while the number variance of $\mathcal{X}_{\psi }(\Omega )$ is (see \cite[%
pg. 40]{Variance} for a detailed proof):%
\begin{equation*}
\mathbb{V}\left[ \mathcal{X}_{\psi }(\Omega )\right] =\mathbb{E}\left( 
\mathcal{X}_{\psi }(\Omega )^{2}\right) -\mathbb{E}\left( \mathcal{X}_{\psi
}(\Omega )\right) ^{2}=trace\left( T_{\Omega }\right) -trace\left( T_{\Omega
}^{2}\right) \text{.}
\end{equation*}

Our first result gives the asymptotic behavior of the variance and the exact
value of the asymptotic constant. A related formula has been obtained by
Shirai for Ginibre ensembles in higher Landau levels \cite{Shirai}. A new
proof of Shirai's formula has been obtained by Demni and Lazag \cite%
{HyperbolicDPP}, using a quite flexible argument based on geometric
considerations, which inspired the following result.

\begin{theorem}
Let $\psi $ admissible with $\left\Vert \psi \right\Vert _{H^{2}\left( 
\mathbb{C}^{+}\right) }^{2}=1$. We have the following explicit formula for
the variance of the affine ensemble associated with $\psi $:%
\begin{equation*}
\mathbb{V}\left[ \mathcal{X}_{\psi }(D(i,R))\right] =\int_{\mathbb{C}%
^{+}}\left\vert W_{\psi }\psi ({w})\right\vert ^{2}\left\vert D(i,R)^{c}\cap
D(w,R)\right\vert _{h}d\mu ^{+}(w)\text{.}
\end{equation*}%
Moreover, as $R\rightarrow 1^{-}$,%
\begin{equation}
\mathbb{V}\left[ \mathcal{X}_{\psi }(D(i,R))\right] \sim \frac{c_{\psi }}{%
1-R^{2}}\text{,}  \label{asymptotic}
\end{equation}%
where the asymptotic constant $c_{\psi }$ is given by%
\begin{equation}
c_{\psi }=\frac{1}{2}\int_{\mathbb{C}^{+}}\left\vert W_{\psi }\psi ({w}%
)\right\vert ^{2}\arccos \left( 1-2\left\vert \frac{w-i}{w+i}\right\vert
^{2}\right) d\mu ^{+}(w)\text{.}  \label{constant}
\end{equation}
\end{theorem}

\begin{proof}
In the context of the concentration operator defined in the beginning of the
section, set $\Omega =D(i,R)$. Observe that $K_{\psi }\left( z,z\right)
=W_{\psi }\psi ({z\cdot z}^{-1})=W_{\psi }\psi ({i})$, that $W_{\psi }\psi ({%
w}^{-1}\cdot z)=\left\langle \pi (w)\psi ,\pi ({z})\psi \right\rangle
_{H^{2}(\mathbb{C}^{+})}=\overline{W_{\psi }\psi ({z}^{-1}\cdot {w})}$, and
use the reproducing kernel equation, $d\mu ^{+}$ as the left Haar measure on
the $ax+b$ group, and Fubini, to write:%
\begin{eqnarray*}
&&\mathbb{V}\left[ \mathcal{X}_{\psi }(D(i,R))\right]  \\
&=&\int_{D(i,R)}W_{\psi }\psi ({i})d\mu ^{+}(z)-\int_{D(i,R)\times
D(i,R)}\left\vert W_{\psi }\psi ({w}^{-1}\cdot {z})\right\vert ^{2}d\mu
^{+}(w)d\mu ^{+}(z) \\
&=&\int_{D(i,R)\times \mathbb{C}^{+}}\left\vert W_{\psi }\psi ({z}^{-1}\cdot 
{w})\right\vert ^{2}d\mu ^{+}(w)d\mu ^{+}(z)-\int_{D(i,R)\times
D(i,R)}\left\vert W_{\psi }\psi ({z}^{-1}\cdot {w})\right\vert ^{2}d\mu
^{+}(w)d\mu ^{+}(z) \\
&=&\int_{D(i,R)\times D(i,R)^{c}}\left\vert W_{\psi }\psi ({z}^{-1}\cdot {w}%
)\right\vert ^{2}d\mu ^{+}(w)d\mu ^{+}(z) \\
&=&\int_{\mathbb{C}^{+}}1_{D(i,R)^{c}}(z)\left[ \int_{\mathbb{C}%
^{+}}1_{D(i,R)}(w)\left\vert W_{\psi }\psi ({z}^{-1}\cdot {w})\right\vert
^{2}d\mu ^{+}(w)\right] d\mu ^{+}(z) \\
&=&\int_{\mathbb{C}^{+}}1_{D(i,R)^{c}}(z)\left[ \int_{\mathbb{C}%
^{+}}1_{D(i,R)}(z\cdot w)\left\vert W_{\psi }\psi ({w})\right\vert ^{2}d\mu
^{+}(w)\right] d\mu ^{+}(z) \\
&=&\int_{\mathbb{C}^{+}}1_{D(i,R)^{c}}(z)\left[ \int_{\mathbb{C}%
^{+}}1_{D(w^{-1},R)}(z)\left\vert W_{\psi }\psi ({w})\right\vert ^{2}d\mu
^{+}(w)\right] d\mu ^{+}(z) \\
&=&\int_{\mathbb{C}^{+}}\left\vert W_{\psi }\psi ({w})\right\vert ^{2}\left[
\int_{D(i,R)^{c}\cap D(w^{-1},R)}d\mu ^{+}(z)\right] d\mu ^{+}(w)\text{,}
\end{eqnarray*}%
where $1_{D(i,R)}(z\cdot w)=1_{D(w^{-1},R)}(z)$ follows from $\varrho
(z.w,i)=\varrho (w^{-1},z)$. Since $D(w^{-1},R)=D(w,R)$,\ we conclude that%
\begin{equation}
\mathbb{V}\left[ \mathcal{X}_{\psi }(D(i,R))\right] =\int_{\mathbb{C}%
^{+}}\left\vert W_{\psi }\psi ({w})\right\vert ^{2}\left\vert D(i,R)^{c}\cap
D(w,R)\right\vert _{h}d\mu ^{+}(w)\text{.}  \label{V1}
\end{equation}%
To prove (\ref{asymptotic}) and to determine the asymptotic constant (\ref%
{constant}), we will need to find the area $\left\vert D(i,R)^{c}\cap
D(w^{-1},R)\right\vert _{h}$ when $R\rightarrow 1$. First move the integrals
conformal to the unit disc. Setting%
\begin{equation*}
\xi _{(z)}=\frac{z-i}{z+i}\in \mathbb{D};\text{ \ \ \ \ \ \ }\xi _{(w)}^{-1}=%
\frac{w+1}{i(w-1)}\in \mathbb{C}^{+}\text{,}
\end{equation*}%
the measures can be related by 
\begin{equation*}
(\mathrm{Im}\,z)^{\alpha }d\mu _{\mathbb{C}^{+}}(z)=\frac{2^{\alpha
+1}(1-\left\vert \xi _{(z)}\right\vert ^{2})^{\alpha }}{(1-\xi
_{(z)})^{2\alpha +4}}d\mu _{\mathbb{D}}(\xi _{(z)})\text{,}
\end{equation*}%
$d\mu _{\mathbb{D}}$ being the Lesbegue measure on $\mathbb{D}$. Denoting by 
$\mathbb{D}(\xi _{(w)},R)$ the hyperbolic disc on $\mathbb{D}$ resulting
from conformal mapping $D(w,R)$ we have%
\begin{equation*}
\int_{D(i,R)^{c}\cap D(w,R)}\frac{1}{\left( \mathrm{Im}\,z\right) ^{2}}d\mu
_{\mathbb{C}^{+}}(z)=\frac{1}{2}\int_{\mathbb{D}(0,R)^{c}\cap \mathbb{D}(\xi
_{(w)},R)}(1-\left\vert \xi _{(z)}\right\vert ^{2})^{-2}d\mu _{\mathbb{D}%
}(\xi _{(z)})
\end{equation*}%
and we can use the computation of Theorem 1 in \cite{HyperbolicDPP}, leading
to, as $R\rightarrow 1^{-}$,%
\begin{equation*}
\left\vert D(i,R)^{c}\cap D(w,R)\right\vert _{h}=\frac{1}{2}\left\vert
\left( \mathbb{D}(0,R)^{c}\cap \mathbb{D}(\xi _{(w)},R)\right) \right\vert
_{h}\sim \frac{1}{2}\frac{\arccos (1-2\left\vert \xi _{(w)}\right\vert ^{2})%
}{1-R^{2}}\text{.}
\end{equation*}%
Thus, 
\begin{equation*}
\left\vert D(i,R)^{c}\cap D(w,R)\right\vert _{h}\sim \frac{1}{2}\frac{%
\arccos (1-2\left\vert \frac{w-i}{w+i}\right\vert ^{2})}{1-R^{2}}
\end{equation*}%
It follows that, as $R\rightarrow 1^{-}$, 
\begin{equation*}
\mathbb{V}\left[ \mathcal{X}_{\psi }(D(i,R))\right] \sim \frac{1}{2}\frac{1}{%
1-R^{2}}\int_{\mathbb{C}^{+}}\left\vert W_{\psi }\psi ({w})\right\vert
^{2}\arccos (1-2\left\vert \frac{w-i}{w+i}\right\vert ^{2})d\mu ^{+}(w)\text{%
.}
\end{equation*}
\end{proof}

The next result shows with a two-sided inequality that the variance of the
affine ensemble is proportional to $\left\vert \Omega \right\vert _{h}$. The
first part of the result is essentially an interpretation of the results in 
\cite{DeMarie}, where the lower inequality is obtained for a large class of
sets $\Omega $, assuming that, for some $c>0$, 
\begin{equation}
\left\vert \left\langle \pi ({z})\psi ,\pi ({w})\psi \right\rangle _{H(%
\mathbb{C}^{+})}\right\vert ^{2}\geq \frac{c}{\left\vert D(z,R)\right\vert
_{h}^{2}}\int_{\mathbb{C}^{+}}1_{D(\xi ,R)}(z)1_{D(\xi ,R)}(w)d\xi \text{.}
\label{two}
\end{equation}%
For $\Omega =D(i,R)$ we provide a proof of an upper bound involving the
admissibility constant $C_{\psi }$.

\begin{theorem}
Assuming that (\ref{two}) holds, we have%
\begin{equation}
\left\vert \Omega \right\vert _{h}\lesssim \mathbb{V}\left[ \mathcal{X}%
_{\psi }(\Omega )\right] \leq \left\vert \Omega \right\vert _{h}\text{.}
\label{cond}
\end{equation}%
If $\Omega =D(i,R)$, then 
\begin{equation*}
\mathbb{V}\left[ \mathcal{X}_{\psi }(D(i,R))\right] \leq C_{\psi }\left\vert
D(i,R)\right\vert _{h}\text{.}
\end{equation*}
\end{theorem}

\begin{proof}
Using the operator $T_{\Omega }$ we easily obtain an upper bound for the
variance 
\begin{equation*}
\mathbb{V}\left[ \mathcal{X}_{\psi }(\Omega )\right] =trace\left( T_{\Omega
}\right) -trace\left( T_{\Omega }^{2}\right) \leq trace\left( T_{\Omega
}\right) =\left\vert \Omega \right\vert _{h}\text{,}
\end{equation*}%
since $trace\left( T_{\Omega }^{2}\right) \leq trace\left( T_{\Omega
}\right) $. The lower inequality follows from \cite[Lemma 3.2]{DeMarie},
where it is shown that, under the condition (\ref{two}),%
\begin{equation*}
trace\left( T_{\Omega }\right) -trace\left( T_{\Omega }^{2}\right) \gtrsim
\left\vert \Omega \right\vert _{h}\text{.}
\end{equation*}%
Now set $\Omega =D(i,R)$. Then, from Theorem 1, 
\begin{eqnarray*}
\mathbb{V}\left[ \mathcal{X}_{\psi }(D(i,R))\right] &=&\int_{\mathbb{C}%
^{+}}\left\vert W_{\psi }\psi ({w})\right\vert ^{2}\left[ \int_{D(i,R)^{c}%
\cap D(w,R)}d\mu ^{+}(z)\right] d\mu ^{+}(w) \\
&\leq &\int_{\mathbb{C}^{+}}\left\vert W_{\psi }\psi ({w})\right\vert
^{2}\left\vert D(w,R)\right\vert _{h}d\mu ^{+}(w) \\
&=&\left\vert D(i,R)\right\vert _{h}\int_{\mathbb{C}^{+}}\left\vert W_{\psi
}\psi ({w})\right\vert ^{2}d\mu ^{+}(w) \\
&=&C_{\psi }\left\vert D(i,R)\right\vert _{h}\text{,}
\end{eqnarray*}%
using $\left\vert D(w,R)\right\vert _{h}=\left\vert D(i,R)\right\vert _{h}$\
and (\ref{Cg}).
\end{proof}

\begin{remark}
A reduced DPP adapted to $\Omega $ can be defined as a Toeplitz smooth
restriction of the affine ensemble to $\Omega $ using the operator (\ref%
{Toeplitz}), following a scheme similar to the one used to define the finite
Weyl-Heisenberg ensembles in \cite{abgrro17}. Denoting by $\{p_{j}^{\Omega
}\}$ the eigenfunctions of (\ref{Toeplitz}), one associates with $\Omega $\
the reduced finite dimensional Hilbert space%
\begin{equation*}
W_{\psi }^{N_{\Omega }}=Span\{p_{j}^{\Omega }\}_{n=1,...,N_{\Omega }}\subset
W_{\psi }\text{,}
\end{equation*}%
where $N_{\Omega }$ $=\left\lfloor \left\vert \Omega \right\vert
_{h}\right\rfloor $, the least integer than or equal to $\left\vert \Omega
\right\vert _{h}$. The $\Omega $-reduced affine ensemble is the finite
dimensional DPP $\mathcal{X}_{\psi }^{\Omega }$\ generated by the kernel%
\begin{equation*}
K_{\psi ,\Omega }(z,w)=\sum_{j=0}^{N_{\Omega }}p_{j}^{\Omega }(z)\overline{%
p_{j}^{\Omega }(w)}\text{.}
\end{equation*}
\end{remark}

\section{The Maass-Landau levels processes}

\subsection{Bergman spaces}

The reproducing kernel of the space $W_{\widetilde{\psi _{0}^{\alpha }}%
}\left( H^{2}\left( \mathbb{C}^{+}\right) \right) $ is the following
weighted Bergman kernel (take $n=0$ in Proposition 1 in the Appendix):%
\begin{equation*}
K_{_{\widetilde{\psi _{0}^{\alpha }}}}(z,w)=\alpha \left( 4\,\mathrm{Im}\,z\,%
\mathrm{Im}\,w\right) ^{\frac{\alpha +1}{2}}\left( \frac{1}{-i(z-\overline{w}%
)}\right) ^{\alpha +1}\text{.}
\end{equation*}%
This is the `ground level' case of the structure considered in the next
section. For $\alpha =1$ it is a $\mathbb{C}^{+}$ weighted version of the
DPP studied by Peres and Vir\'{a}g \cite{PeresVirag}.

\subsection{Hyperbolic Maass-Landau levels}

The Hamiltonian describing the dynamics of a charged particle moving on the
Poincar\'{e} upper half-plane $\mathbb{C}^{+}$ under the action of the
magnetic field $B$ is given by : 
\begin{equation*}
H_{B}:=s^{2}\left( \frac{\partial ^{2}}{\partial x^{2}}+\frac{\partial ^{2}}{%
\partial s^{2}}\right) -2iBs\frac{\partial }{\partial x}
\end{equation*}%
The operator $H_{B}$ was first introduced by Maass in number theory \cite%
{Maass,Patterson} and its interpretation as a hyperbolic analogue of the
Landau Hamiltonian has been put forward by Comtet (see \cite{Comtet,AoP}).
We list here the following important properties of $H_{B}$ as an operator.

\begin{enumerate}
\item $H_{B}$ is an elliptic densely defined operator on the Hilbert space $%
L^{2}\left( \mathbb{C}^{+},d\mu ^{+}\right) $, with a unique self-adjoint
realization that we denote also by $H_{B}$.

\item The spectrum\ of $H_{B}$ in $L^{2}\left( \mathbb{C}^{+},d\mu
^{+}\right) $ consists of two parts:\textit{\ }a continuous part $\left[
1/4,+\infty \right[ $, corresponding to \textit{scattering states} and a
finite number of eigenvalues with infinite degeneracy (\textit{hyperbolic
Landau levels}) of the form 
\begin{equation}
\epsilon _{n}^{B}:=(B-n)\left( 1-B+n\right) ,n=0,1,2,\cdots ,\lfloor B-\frac{%
1}{2}\rfloor \text{.}  \label{3.1.2}
\end{equation}%
The finite part of the spectrum exists provided $2B>1$. The notation $%
\lfloor a\rfloor $ stands for the greatest integer not exceeding $a.$

\item For each fixed eigenvalue $\epsilon _{n}^{B}$, we denote by 
\begin{equation}
\mathcal{E}_{n}^{B}\left( \mathbb{C}^{+}\right) =\left\{ F\in L^{2}\left( 
\mathbb{C}^{+},d\mu ^{+}\right) ,H_{B}F=\epsilon _{n}^{B}F\right\} 
\label{3.1.3}
\end{equation}%
the corresponding eigenspace, which has a reproducing kernel given by 
\begin{eqnarray*}
&&K_{n,B}\left( z,w\right)  \\
&=&\frac{\left( -1\right) ^{n}\Gamma \left( 2B-n\right) }{n!\Gamma \left(
2B-2n\right) }\left( \frac{4\,\mathrm{Im}\,z\,\mathrm{Im}\,w}{\left\vert z-%
\overline{w}\right\vert ^{2}}\right) ^{B-n}\text{ }\left( \frac{\overline{z}%
-w}{\overline{w}-z}\right) ^{B}F\left[ 
\begin{array}{c}
-2B-n,-n \\ 
2B-2n%
\end{array}%
;\frac{4\,\mathrm{Im}\,z\,\mathrm{Im}\,w}{\left\vert z-\overline{w}%
\right\vert ^{2}}\right] \text{,}
\end{eqnarray*}%
where $F$ is the Gauss hypergeometric function:%
\begin{equation*}
F\left[ 
\begin{array}{c}
a,b \\ 
c%
\end{array}%
;z\right] =_{2}F_{1}\left[ 
\begin{array}{c}
a,b \\ 
c%
\end{array}%
;z\right] =\sum_{n=0}^{\infty }\frac{(a)_{n}(b)_{n}}{n!(c)_{n}}z^{n}\text{.}
\end{equation*}
\end{enumerate}

The condition $2B>1$ ensuring the existence of these discrete eigenvalues in 
$\left( 2.\right) $ means that the magnetic field has to be strong enough to
capture the particle in a closed orbit. If this condition is not fulfilled,
the motion will be unbounded and the orbit of the particle will intercept
the upper half-plane boundary whose points stand for `points at infinity'\
(see \cite[p. 189]{Comtet}). \ The eigenvalues in $\left( 2.\right) $ which
are below the continuous spectrum have eigenfunctions called \textit{bound
states }since the particle in such a state cannot leave the system without
additional energy. Then the number of particle layers (Landau levels), $%
\lfloor B-\frac{1}{2}\rfloor $, depends on the strength $B$ of the magnetic
field.\ 

To make the identification with special cases of the affine ensemble kernel,
in the Appendix we compute the reproducing kernels and admissibility
constants of the spaces $W_{\widetilde{\psi _{n}^{\alpha }}}\left(
H^{2}\left( \mathbb{C}^{+}\right) \right) $ in terms of hypergeometric
functions and Jacobi polynomials. According to Proposition 1, the
reproducing kernel of $\mathcal{W}_{\widetilde{\psi _{n}^{2(B-n)-1}}}=W_{%
\widetilde{\psi _{n}^{\alpha }}}\left( H^{2}\left( \mathbb{C}^{+}\right)
\right) $ is given by%
\begin{equation*}
K_{\widetilde{\psi _{n}^{2(B-n)-1}}}(z,w)=K_{n,B}\left( z,w\right) \text{.}
\end{equation*}%
Thus, the kernels \ $K_{_{\widetilde{\psi _{n}^{2(B-n)-1}}}}(z,w)$, are
precisely the reproducing kernels of the eigenspaces associated with the
pure point spectrum of the Maass Laplacian. As a result, all properties of
the affine ensemble are automatically translated to the DPP associated with
the reproducing kernels $K_{\psi _{n}^{2(B-n)-1}}(z,w)$, with asymptotic
constant%
\begin{eqnarray*}
c_{\widetilde{\psi _{n}^{2(B-n)-1}}} &=&\frac{1}{2}\int_{\mathbb{C}%
^{+}}\left\vert W_{\widetilde{\psi _{n}^{2(B-n)-1}}}\widetilde{\psi
_{n}^{2(B-n)-1}}({w})\right\vert ^{2}\arccos (1-2\left\vert \frac{w-i}{w+i}%
\right\vert ^{2})d\mu ^{+}(w) \\
&=&\frac{1}{2}\int_{\mathbb{C}^{+}}\left\vert K_{n,B}(i,w)\right\vert
^{2}\arccos (1-2\left\vert \frac{w-i}{w+i}\right\vert ^{2})d\mu ^{+}(w)
\end{eqnarray*}%
and admissibility constant%
\begin{equation*}
C_{\widetilde{\psi _{n}^{2(B-n)-1}}}=\frac{4\pi }{2(B-n)-1}\text{.}
\end{equation*}%
Then, Theorems 1 and 2 lead to the following result for \emph{the
Maass-Landau process, the DPP }$\mathcal{X}_{B,n}$\emph{\ generated by the
reproducing kernel }$K_{n,B}\left( z,w\right) $\emph{\ of the eigenspace of }%
$H_{B}$\emph{\ associated with the Maass-Landau level eigenvalue }$\epsilon
_{n}^{B}:=(B-n)\left( 1-B+n\right) $, for $n=0,1,2,\cdots ,\lfloor B-\frac{1%
}{2}\rfloor $.

\begin{corollary}
The variance of $\mathcal{X}_{B,n}$ is given by 
\begin{equation*}
\mathbb{V}\left[ \mathcal{X}_{B,n}(D(i,R))\right] =\int_{\mathbb{C}%
^{+}}\left\vert K_{n,B}(i,w)\right\vert ^{2}\left\vert D(i,R)^{c}\cap
D(w,R)\right\vert _{h}d\mu ^{+}(w)\text{.}
\end{equation*}%
Moreover, when $R\rightarrow 1^{-}$,%
\begin{equation*}
\mathbb{V}\left[ \mathcal{X}_{B,n}(D(i,R))\right] \sim \frac{c_{n,B}}{1-R^{2}%
}\text{,}
\end{equation*}%
where$\ c_{n,B}=\frac{1}{2}\int_{\mathbb{C}^{+}}\left\vert
K_{n,B}(i,w)\right\vert ^{2}\arccos (1-2\left\vert \frac{w-i}{w+i}%
\right\vert ^{2})d\mu ^{+}(w)$. Finally, the variance of $\mathcal{X}_{B,n}$
satisfies the non-asymptotic bound%
\begin{equation*}
\mathbb{V}\left[ \mathcal{X}_{B,n}(D(i,R))\right] \leq \left( \frac{4\pi }{%
2(B-n)-1}\right) \left\vert D(i,R)\right\vert _{h}\text{.}
\end{equation*}
\end{corollary}

\section{Appendix}

\subsection{Reproducing kernels of special affine ensembles}

Let $\alpha >-1$. For $n=0,1,2,\dots $ we define the normalized functions $%
\widetilde{\psi _{n}^{\alpha }}$ such that $\left\Vert \widetilde{\psi
_{n}^{\alpha }}\right\Vert _{H^{2}\left( \mathbb{C}^{+}\right) }=1$: 
\begin{equation*}
(\mathcal{F}\widetilde{\psi _{n}^{\alpha }})(t)=\sqrt{\frac{2^{\alpha +1}n!}{%
\Gamma (n+\alpha +1)}}t^{\frac{\alpha }{2}}e^{-t}L_{n}^{\alpha }(2t),\quad
t>0.
\end{equation*}

\begin{proposition}
The reproducing kernel of $\mathcal{W}_{\widetilde{\psi _{n}^{\alpha }}}$ is
given by 
\begin{eqnarray*}
K_{\widetilde{\psi _{n}^{\alpha }}}(z,w) &=&\frac{(-1)^{n}\Gamma (n+1+\alpha
)}{n!\Gamma (1+\alpha )}\left( \frac{4\,\mathrm{Im}\,z\,\mathrm{Im}\,w}{|z-%
\overline{w}|^{2}}\right) ^{\frac{\alpha +1}{2}}\left( \frac{\overline{z}-w}{%
\overline{w}-z}\right) ^{\frac{\alpha +1}{2}+n} \\
&&\times F\left[ 
\begin{array}{c}
n+\alpha +1,-n \\ 
1+\alpha%
\end{array}%
;\frac{4\,\mathrm{Im}\,z\,\mathrm{Im}\,w}{|z-\overline{w}|^{2}}\right] \text{%
,}
\end{eqnarray*}%
where $F=_{2}F_{1}$ denotes the hypergeometric function. Setting $\alpha
=2(B-n)-1$ we obtain 
\begin{eqnarray*}
&&K_{\widetilde{\psi _{n}^{2(B-n)-1}}}(z,w) \\
&=&\frac{(-1)^{n}\Gamma (2B-n)}{n!\Gamma (2B-2n)}\left( \frac{4\,\mathrm{Im}%
\,z\,\mathrm{Im}\,w}{|z-\overline{w}|^{2}}\right) ^{B-n}\left( \frac{%
\overline{z}-w}{\overline{w}-z}\right) ^{B}F\left[ 
\begin{array}{c}
2B-n,-n \\ 
2B-2n%
\end{array}%
;\frac{4\,\mathrm{Im}\,z\,\mathrm{Im}\,w}{|z-\overline{w}|^{2}}\right]
\end{eqnarray*}
\end{proposition}

\begin{proof}
For simplification write $z=x+is,w=x^{\prime }+is^{\prime }\in \mathbb{C^{+}}
$. Formula (\ref{kernelF}) gives:%
\begin{eqnarray*}
K_{\widetilde{\psi _{n}^{\alpha }}}(z,w) &=&\langle \pi (w)\widetilde{\psi
_{n}^{\alpha }},\pi (z)\widetilde{\psi _{n}^{\alpha }}\rangle _{H^{2}(%
\mathbb{C}^{+})}=\left\langle \widehat{\pi (w)\widetilde{\psi _{n}^{\alpha }}%
},\widehat{\pi (z)\widetilde{\psi _{n}^{\alpha }}}\right\rangle _{L^{2}(%
\mathbb{R}^{+})} \\
&=&\frac{2^{\alpha +1}n!}{\Gamma (n+\alpha +1)}s^{\frac{1}{2}}s^{\prime 
\frac{1}{2}}\int_{0}^{\infty }e^{-ix^{\prime }t}(ts^{\prime })^{\frac{\alpha 
}{2}}e^{-s^{\prime }t}L_{n}^{\alpha }(2s^{\prime }t)e^{ixt}(ts)^{\frac{%
\alpha }{2}}e^{-st}L_{n}^{\alpha }(2st)dt \\
&=&\frac{n!}{\Gamma (n+\alpha +1)}s^{\frac{\alpha }{2}+\frac{1}{2}}s^{\prime 
\frac{\alpha }{2}+\frac{1}{2}}\int_{0}^{\infty }t^{\alpha }e^{-t\Big(\frac{1%
}{2}i(x^{\prime }-x)+\frac{s^{\prime }+s}{2}\Big)}L_{n}^{\alpha }(s^{\prime
}t)L_{n}^{\alpha }(st)dt\text{.}
\end{eqnarray*}%
To compute the integral we will use the following integral formula \cite[p.
810, 7.414 (13)]{tables}:%
\begin{equation*}
\int_{0}^{\infty }e^{-t(k+\frac{a_{1}+a_{2}}{2})}t^{\alpha }L_{n}^{\alpha
}(a_{1}t)L_{n}^{\alpha }(a_{2}t)dt=\frac{\Gamma (1+\alpha +n)}{%
b_{0}^{1+\alpha +n}}\frac{b_{2}^{n}}{n!}P_{n}^{(\alpha ,0)}\left( \frac{%
b_{1}^{2}}{b_{0}b_{2}}\right) ,
\end{equation*}%
where 
\begin{equation*}
b_{0}=k+\frac{a_{1}+a_{2}}{2},b_{2}=k-\frac{a_{1}+a_{2}}{2}%
,b_{1}^{2}=b_{0}b_{2}+2a_{1}a_{2},\mathrm{Re}\,\alpha >-1,\mathrm{Re}%
\,b_{0}>0
\end{equation*}%
and $P_{n}^{(\alpha ,\beta )}$ denotes the Jacobi polynomial. Setting 
\begin{equation*}
b_{0}=\frac{1}{2}i(x^{\prime }-x)+s^{\prime }+s=\frac{1}{2}i(x^{\prime
}-x-is^{\prime }-is)=\frac{1}{2}i(x^{\prime }-is^{\prime }-(x+is))=\frac{1}{2%
}i(\overline{w}-z),
\end{equation*}%
\begin{equation*}
b_{2}=\frac{1}{2}i(x^{\prime }-x)-s^{\prime }-s=\frac{1}{2}i(x^{\prime
}-x+is^{\prime }+is)=\frac{1}{2}i(x^{\prime }+is^{\prime }-(x-is))=\frac{1}{2%
}i(w-\bar{z})
\end{equation*}%
and 
\begin{equation*}
b_{1}^{2}=\frac{1}{4}(-|z-\bar{w}|^{2}+8s^{\prime }s).
\end{equation*}%
gives%
\begin{equation*}
K_{\widetilde{\psi _{n}^{\alpha }}}(z,w)=(ss^{\prime })^{\frac{\alpha }{2}+%
\frac{1}{2}}\Big(\frac{2}{i}\Big)^{\alpha +1}\frac{(w-\overline{z})^{n}}{(%
\overline{w}-z)^{\alpha +n+1}}\cdot P_{n}^{(\alpha ,0)}\left( 1-\frac{%
8s^{\prime }s}{|z-\overline{w}|^{2}}\right) \text{,}
\end{equation*}%
where the Jacobi polynomial is defined as 
\begin{equation}
P_{n}^{(\alpha ,\beta )}(x)=\frac{\Gamma (n+1+\alpha )}{n!\Gamma (1+\alpha )}%
F\left[ 
\begin{array}{c}
n+\alpha +\beta +1,-n \\ 
1+\alpha%
\end{array}%
;\frac{1-x}{2}\right] \text{.}  \label{Jacobi}
\end{equation}%
Thus, 
\begin{eqnarray*}
&&K_{\widetilde{\psi _{n}^{\alpha }}}(z,w) \\
&=&(4ss^{\prime })^{\frac{\alpha }{2}+\frac{1}{2}}\frac{1}{i^{\alpha +1}}%
\frac{\Gamma (n+\alpha +1)}{n!\Gamma (\alpha +1)}\frac{(w-\overline{z})^{n}}{%
(\overline{w}-z)^{\alpha +n+1}}F\left[ 
\begin{array}{c}
n+1+\alpha ,-n \\ 
1+\alpha%
\end{array}%
;\frac{4s^{\prime }s}{|z-\overline{w}|^{2}}\right] \text{.}
\end{eqnarray*}%
Next, notice that 
\begin{eqnarray*}
&&\frac{1}{(|z-\overline{w}|^{2})^{\frac{\alpha }{2}+\frac{1}{2}}}\frac{(w-%
\overline{z})^{n}(z-\overline{w})^{\frac{\alpha }{2}+\frac{1}{2}}(\overline{z%
}-w)^{\frac{\alpha }{2}+\frac{1}{2}}}{(\overline{w}-z)^{\alpha +n+1}} \\
&=&\frac{1}{(|z-\overline{w}|^{2})^{\frac{\alpha }{2}+\frac{1}{2}}}\frac{%
(-1)^{n}(\overline{z}-w)^{n+\frac{\alpha }{2}+\frac{1}{2}}(-1)^{\frac{\alpha 
}{2}+\frac{1}{2}}(\overline{w}-z)^{\frac{\alpha }{2}+\frac{1}{2}}}{(%
\overline{w}-z)^{\alpha +n+1}}\text{.}
\end{eqnarray*}%
Hence, 
\begin{eqnarray*}
&&K_{\widetilde{\psi _{n}^{\alpha }}}(z,w) \\
&=&\frac{(-1)^{n}\Gamma (n+\alpha +1)}{n!\Gamma (\alpha +1)}\left( \frac{%
4s^{\prime }s}{|z-\overline{w}|^{2}}\right) ^{\frac{\alpha }{2}+\frac{1}{2}%
}\left( \frac{\overline{z}-w}{\overline{w}-z}\right) ^{\frac{\alpha }{2}+%
\frac{1}{2}+n}F\left[ 
\begin{array}{c}
n+1+\alpha ,-n \\ 
1+\alpha%
\end{array}%
;\frac{4s^{\prime }s}{|z-\overline{w}|^{2}}\right] \text{.}
\end{eqnarray*}
\end{proof}

\subsection{\textbf{Norms and admissible constants}}

The orthogonality relation%
\begin{equation}
\int_{0}^{+\infty }L_{n}^{\alpha }(t)L_{m}^{\alpha }(t)t^{\alpha }e^{-t}dt=%
\frac{\Gamma (n+\alpha +1)}{n!}\delta _{n,m}  \label{ortogonal}
\end{equation}%
written as%
\begin{equation*}
\int_{0}^{+\infty }t^{\frac{\alpha }{2}}e^{-t}L_{n}^{\alpha }(2t)t^{\frac{%
\alpha }{2}}e^{-t}L_{m}^{\alpha }(2t)dt=\frac{\Gamma (n+\alpha +1)}{%
2^{\alpha +1}n!}\delta _{n,m}
\end{equation*}%
shows that $\mathcal{F}\widetilde{\psi _{n}^{\alpha }}\in L^{2}(\mathbb{R}%
^{+})$. Hence, we observe that $\widetilde{\psi _{n}^{\alpha }}\in H^{2}(%
\mathbb{C}^{+})$ and 
\begin{equation*}
\left\Vert \widetilde{\psi _{n}^{\alpha }}\right\Vert _{H^{2}(\mathbb{C}%
^{+})}^{2}=\left\Vert \mathcal{F}\widetilde{\psi _{n}^{\alpha }}\right\Vert
_{L^{2}(%
\mathbb{R}
^{+})}^{2}=1\text{.}
\end{equation*}%
Next, we calculate the admissibility constant $C_{\widetilde{\psi
_{n}^{\alpha }}}$ . The formula \cite[(10)]{Sriv} gives 
\begin{equation*}
\int_{0}^{\infty }t^{\alpha -1}e^{-t}L_{n}^{\alpha }(2t)^{2}dt=\frac{\Gamma
(n+\alpha +1)}{2^{\alpha }n!\alpha }
\end{equation*}%
\begin{equation*}
C_{\widetilde{\psi _{n}^{\alpha }}}=2\pi \Vert \mathcal{F}\widetilde{\psi
_{n}^{\alpha }}\Vert _{L^{2}(\mathbb{R}^{+},t^{-1})}=\frac{\pi 2^{\alpha
+2}n!}{\Gamma (n+\alpha +1)}\int_{0}^{\infty }\left( t^{\alpha
/2}e^{-t}L_{n}^{\alpha }(2t)\right) ^{2}\frac{dt}{t}=\frac{4\pi }{\alpha }%
\text{.}
\end{equation*}%
\noindent Consequently, 
\begin{equation*}
C_{\widetilde{\psi _{n}^{2(B-n)-1}}}=\frac{2}{2(B-n)-1}\text{.}
\end{equation*}

\end{document}